\documentclass[Royal,sagev,times]{sagej}

\usepackage{amsmath,amssymb,amsthm}
\usepackage{graphicx}
\usepackage[colorlinks,bookmarksopen,bookmarksnumbered,citecolor=blue,urlcolor=blue]{hyperref}

\newcommand{\R}{\mathbb{R}}
\newcommand{\C}{\mathbb{C}}

\theoremstyle{plain}
\newtheorem{theorem}{Theorem}[section]
\newtheorem{lemma}[theorem]{Lemma}

\theoremstyle{definition}

\numberwithin{equation}{section}

\makeatletter
\def\@maketitle{%
  \null
  \vskip -2em% 适当调整顶部间距
  \begin{center}%
    \let\footnote\thanks
    {\LARGE \textbf{\@title} \par}%
    \vskip 1.5em%
    {\large
     \lineskip .5em%
     \begin{tabular}[t]{c}%
       \@author
     \end{tabular}\par}%
  \end{center}%
  \par
  \vskip 1.5em%
  \noindent\usebox\absbox
  \par
  \vskip 1em%
  \noindent\textbf{}\ \@keywords
  \par
  \vskip 2em%
}
\makeatother

\makeatletter
\def\ps@sagepage{%
  \let\@mkboth\@gobbletwo
  \def\@evenhead{%
    \parbox{\textwidth}{%
      \normalsize\sagesf\thepage\hfill\\[-6pt]%
      \noindent\rule{\textwidth}{0.25pt}%
    }%
  }%
  \def\@oddhead{%
    \parbox{\textwidth}{%
      \normalsize\sagesf\hfill\thepage\\[-6pt]%
      \noindent\rule{\textwidth}{0.25pt}%
    }%
  }%
  \def\@evenfoot{%
    \parbox[t]{\textwidth}{%
      \scriptsize{\it Prepared using \textsf{\journalclass}}%
    }%
  }%
  \def\@oddfoot{\@evenfoot}%
}
\def\ps@title{%
  \def\@oddhead{%
    \parbox{\textwidth}{%
      \mbox{}\\[-1pt]%
      \noindent\rule{\textwidth}{0.5pt}%
    }%
  }%
  \let\@evenhead\@oddhead
  \def\@oddfoot{%
    \parbox[t]{\textwidth}{%
      \scriptsize{\it Prepared using \textsf{\journalclass}}%
    }%
  }%
  \let\@evenfoot\@oddfoot
}
\makeatother

\makeatletter
\def\ps@sagepage{%
  \let\@mkboth\@gobbletwo
  \def\@evenhead{%
    \parbox{\textwidth}{%
      \normalsize\sagesf\thepage\hfill\\[-6pt]%
      \noindent\rule{\textwidth}{0.25pt}%
    }%
  }%
  \def\@oddhead{%
    \parbox{\textwidth}{%
      \normalsize\sagesf\hfill\thepage\\[-6pt]%
      \noindent\rule{\textwidth}{0.25pt}%
    }%
  }%
  \def\@evenfoot{\parbox[t]{\textwidth}{}}%
  \def\@oddfoot{\@evenfoot}%
}
\def\ps@title{%
  \def\@oddhead{%
    \parbox{\textwidth}{%
      \mbox{}\\[-1pt]%
      \noindent\rule{\textwidth}{0.5pt}%
    }%
  }%
  \let\@evenhead\@oddhead
  \def\@oddfoot{\parbox[t]{\textwidth}{}}%
  \let\@evenfoot\@oddfoot
}
\makeatother
\begin{document}
\pagestyle{sagepage}
% ========== 页眉==========
\runninghead{Wang and Chang}

% ========== 标题 ==========
\title{Normalized solutions for a nonlinear Dirac equation with an inhomogeneous nonlinearity}

\author{Hao Wang\thanks{\texttt{wanghao168@nenu.edu.cn}}\affilnum{1} and Xiaojun Chang\thanks{\texttt{changxj100@nenu.edu.cn}}\affilnum{2}}

\affiliation{
  \affilnum{1}School of Mathematics and Statistics, Northeast Normal University, Changchun 130024, Jilin, PR China\\
  \affilnum{2}School of Mathematics and Statistics \& Center for Mathematics and Interdisciplinary Sciences, Northeast Normal University, Changchun 130024, Jilin, PR China
}

\corrauth{Xiaojun Chang} %School of Mathematics and Statistics \& Center for Mathematics and Interdisciplinary Sciences, Northeast Normal University, Changchun 130024, Jilin, PR China}
\email{changxj100@nenu.edu.cn}

\begin{abstract}
We study the existence of normalized solutions for the nonlinear Dirac equation
\[
\begin{cases}
    -i\sum\limits_{k=1}^3\alpha_k\partial_k u + m\beta u - |x|^{-b}|u|^{p-2}u = \mu u, \quad x\in\mathbb{R}^3, \\
    \int_{\mathbb{R}^3}|u|^2 dx = a,
\end{cases}
\]
where $b\in(0,1)$, $p\in(2,3-b)$, $a>0$ is a prescribed mass, and $\mu\in\mathbb{R}$ is a Lagrange multiplier. For any $b\in(0,1)$ and $p\in(2,3-b)$, we establish the existence of a normalized solution for all sufficiently small masses $a>0$, with $\mu\in(0,m)$ and $u\in H^{\frac12}(\mathbb{R}^3;\mathbb{C}^4)$.
Our results cover the full range of nonlinearities, including mass-subcritical, mass-critical, and mass-supercritical cases.
The main challenges are the strongly indefinite nature of the Dirac operator and the loss of translation invariance caused by the singular weight $|x|^{-b}$. We overcome these difficulties by combining a constrained min-max reduction method with a novel weighted compact embedding in $L^p(\mathbb{R}^3,|x|^{-b}\,\mathrm{d}x;\mathbb{C}^4)$. This approach circumvents the singular potential at the origin and yields a unified existence theory valid in the small-mass regime.
\end{abstract}

\keywords{Nonlinear Dirac equations, Normalized solutions, Singular weight, Inhomogeneous nonlinearity, Variational methods}

% ========== 生成标题栏 ==========

\maketitle

\section{Introduction}
This paper concerns the existence of normalized solutions for the following nonlinear Dirac equation with a singular inhomogeneous nonlinearity:
\begin{equation}\label{1.1}
\begin{cases}
    -i\alpha\cdot\nabla u + m\beta u - |x|^{-b}|u|^{p-2}u = \mu u, \quad x\in\mathbb{R}^3, \\
    \displaystyle\int_{\mathbb{R}^3}|u|^2 dx = a,
\end{cases}
\end{equation}
where $u:\mathbb{R}^3\to\mathbb{C}^4$ is the Dirac spinor field, $m>0$ is the mass of the particle, $\mu\in\mathbb{R}$ serves as a Lagrange multiplier, $a>0$ is a prescribed value, and $b\in(0,1)$, $p\in(2,3-b)$. The operator $\alpha\cdot\nabla$ is defined as $\sum\limits_{k=1}^3\alpha_k\partial_k$ with $\partial_k=\frac{\partial}{\partial x_k}$, where $\alpha_1,\alpha_2,\alpha_3$ and $\beta$ are the $4\times4$ Pauli-Dirac matrices (in $2\times2$ blocks):
\[
\alpha_k=\begin{pmatrix}
0 & \sigma_k \\
\sigma_k & 0
\end{pmatrix},~
\beta=\begin{pmatrix}
I & 0 \\
0 & -I
\end{pmatrix},
\]
with
\[
\sigma_1=\begin{pmatrix}
0 & 1\\
1 & 0
\end{pmatrix},
\sigma_2=\begin{pmatrix}
0 & -i\\
i & 0
\end{pmatrix},
\sigma_3=\begin{pmatrix}
1 & 0\\
0 & -1
\end{pmatrix}.
\]

The Dirac equation stands as a cornerstone of relativistic quantum mechanics \cite{Thaller-1992}, providing a first-order description of relativistic fermions. In its free form, the Dirac operator $H_0 = -i\alpha\cdot\nabla + m\beta$ is obtained from the energy-momentum relation $E^2 = c^2|q|^2 + m^2c^4$ via the correspondence $q \leftrightarrow -i\hbar\nabla$, where $c$ is the speed of light and $\hbar$ is Planck's constant. Throughout this paper, we adopt the natural units $c = \hbar = 1$.
The time-dependent nonlinear Dirac equation, from which \eqref{1.1} is derived by seeking stationary states of the form $\psi(t,x)=e^{-i\mu t}u(x)$, serves as an influential model in atomic, nuclear, and gravitational physics \cite{Esteban-2008}. In these contexts, the nonlinearity is motivated by self-interacting fermion models in quantum electrodynamics and by effective descriptions of matter under extreme conditions \cite{Finkelstein-1951, Ng-2009}.

The study of normalized solutions is physically motivated by the time conservation of the $L^2$-norm, which represents the total particle number or optical power \cite{Kibble-1978, Weinberg-1989}.
Furthermore, the variational characterization of these solutions provides a robust framework for analyzing orbital stability, a strategy extensively validated for the nonlinear Schr\"odinger (NLS) equation \cite{Cazenave-1982, Soave-2020-JFA, Jeanjean-2022}.

Mathematically, problem \eqref{1.1} is known as the
\emph{prescribed mass problem}. Unlike the \emph{fixed-frequency
problem}, where $\mu$ is given a priori \cite{Esteban-1995, Esteban-1996,Bartsch-2006, Ding-2008, Ding-2012, Ding-2015,Esteban-1999, Figueiredo-2017, Borrelli-2019, Wang-2021}, the \emph{prescribed mass problem} requires determining both $u$ and $\mu$ simultaneously. This seemingly minor shift introduces profound analytical challenges specific to  the Dirac equation. The main difficulty stems from the strongly indefinite nature of the free Dirac operator $H_0 = -i\alpha\cdot\nabla + m\beta$,
whose spectrum $\sigma(H_0)=(-\infty,-m]\cup[m,+\infty)$ is purely continuous and unbounded both from below and above.  Consequently, the associated energy functional is neither bounded from below nor above on the $L^2$-sphere constraint, rendering classical variational techniques used for the NLS equation \cite{Jeanjean-1997, Bartsch-2013} inapplicable.

The variational study of normalized solutions for strongly indefinite problems began with the pioneering work of Buffoni and Jeanjean \cite{Buffoni-1993},  who developed a minimax characterization combined with a Lyapunov-Schmidt reduction to treat a semilinear elliptic equation with a spectral gap nonlinearity.
Subsequently, Buffoni, Esteban, and S\'er\'e \cite{Buffoni-2006} adapted an unconstrained penalization technique of Esteban and S\'er\'e  \cite{Esteban-1999} into the \(L^2\)-constrained setting and integrated it with the variational Lyapunov-Schmidt reduction, thereby establishing a robust and flexible framework for studying strongly indefinite problems under a mass constraint that applies to both local and nonlocal nonlinearities.

In recent years, the study of normalized solutions for nonlinear Dirac equations has attracted considerable interest. For the Maxwell-Dirac system, Nolasco \cite{Nolasco-2021}  proved the existence of a normalized solitary wave by implementing a min-max procedure coupled with concentration-compactness arguments. This strategy is to maximize the energy functional on the set of all functions whose positive-energy component belongs to a fixed one-dimensional subspace
(while the whole negative-energy space is free), and then to minimize the resulting maximum over all such subspaces. A different approach was taken by
Ding, Yu, and Zhao \cite{Ding-2023} for the Dirac equation
\[
\begin{cases}
    -ic\sum\limits_{k=1}^3\alpha_k\partial_k u + mc^2\beta u -\omega u = K(x)|u|^{p-2}u, \\[3pt]
    \displaystyle\int_{\mathbb{R}^3}|u|^2 \, dx = 1,
\end{cases}
\]
where $p\in(2,\frac83)$ and $K$ is a nonnegative bounded weight vanishing at infinity. Their proof combines a Lyapunov-Schmidt reduction, a penalization term, and a mountain-pass argument to construct $L^2$-normalized solitary waves. Extending this line of research, Bahrouni et al. \cite{Bahrouni-2023} drew on ideas from \cite{Buffoni-1993} to develop an approach that integrates perturbation methods with a Lyapunov-Schmidt reduction for a broader class of non-autonomous nonlinearities $f(x,|u|)u$. Within this framework, they proved existence of normalized solutions via the combined perturbation-reduction argument and established multiplicity results using
 the Ljusternik-Schnirelmann theorem.
More recently, Coti Zelati and Nolasco \cite{CotiZelati-2023, CotiZelati-2025} extended the variational strategy of \cite{Nolasco-2021} to other relativistic systems: first to the Klein-Gordon-Dirac system, and then to a nonlinear Dirac equation with Soler-type nonlinearities $|\langle u, \beta u\rangle|^{\frac{\alpha}{2}}$ or pure power nonlinearities $|u|^{\alpha}$, where $\alpha\in(2,\frac{8}{3}]$. In both cases, they established the existence of normalized solutions.

Parallel to the investigation of existence and multiplicity, the nonrelativistic limit of normalized solutions has emerged as a topic of independent interest.
Chen, Ding, Guo, and Wang \cite{Chen-2024} first studied this limit for the system
\begin{eqnarray}\label{Dirac-2}
\begin{cases}
    -ic\sum\limits_{k=1}^3\alpha_k\partial_k u + mc^2\beta u-\Gamma\ast(K|u|^\kappa)K|u|^{\kappa-2}u - P|u|^{s-2}u = \omega u, \\
    \displaystyle\int_{\mathbb{R}^3}|u|^2 dx = 1,
\end{cases}
\end{eqnarray}
which involves decaying weights and combines both local and nonlocal nonlinearities, in the mass-subcritical and mass-critical regimes with $\kappa\in [2, \frac{7}{3})$ and $s\in (2, \frac{8}{3}]$.  They proved that any normalized solution of (\ref{Dirac-2}) converges to a solution of the corresponding nonlinear Schr\"{o}dinger equation as $c\to +\infty$.  Subsequently,
Chen, Ding, Guo, and Yu \cite{Chen-2025} extended the existence theory to the mass-supercritical regime for the pure power nonlinearity $|u|^{p-2}u$ with $p\in(\frac{8}{3},3)$.
By restricting the reduced functional to a carefully chosen open subset on which it is bounded below, they proved the existence of energy ground states and established the equivalence between action and energy ground states across the full Sobolev subcritical range $p\in(2,3)$.
For further results on multiplicity, Kerr-type nonlinearities, and generalized Dirac-Coulomb systems, we refer the reader to \cite{Chen-2026,Chen-2025-JDE} and the references therein.

It is worth noting that, in the Schr\"{o}dinger setting, inhomogeneous nonlinearities involving singular weights such as $|x|^{-b}$ have been intensively studied.
Genoud and Stuart \cite{Genoud-2008} established the existence and stability of standing waves for the equation
\[
-\Delta u+\lambda u=V(x)|u|^{p-1}u,\quad x\in\mathbb{R}^N,
\]
where the potential $V$ decays at infinity like $|x|^{-b}$ with $b\in(0,2)$. Their work also provided a sharp Gagliardo-Nirenberg inequality adapted to such singular weight.
Subsequent research has branched into several directions.  Regarding dynamics, Farah \cite{Farah-2016} derived global well-posedness and blow-up criteria for the supercritical inhomogeneous NLS equation. Meanwhile, Saanouni and Shi \cite{Saanouni-2024} developed a local theory in homogeneous Sobolev spaces for the radial fractional inhomogeneous NLS equation with a focusing singular nonlinearity. Within the framework of normalized solutions, Gou \cite{Gou-2024} conducted a variational analysis for combined inhomogeneous nonlinearities. More recently, Gou, Majdoub, and Saanouni \cite{Gou-2026} characterized the scattering versus blow-up dichotomy below the ground state energy for competing inhomogeneous singular NLS equations.

Despite the aforementioned progress for the Schr\"{o}dinger equation, the corresponding problems for the Dirac equation, particularly for normalized solutions, remain largely unexplored.
The locally singular weight $|x|^{-b}$ breaks translation invariance, making standard concentration-compactness arguments inapplicable.
Moreover, the singularity at the origin requires sharper estimates to ensure the energy functional is well defined and to recover compactness, which must be achieved concurrently with handling the Dirac operator's strong indefiniteness. The confluence of these three obstacles, i.e., strong indefiniteness, a locally unbounded nonlinearity, and the loss of translation invariance, is precisely what distinguishes our problem from previous works and constitutes its core mathematical difficulty.  To the best of our knowledge, this paper appears to be the first to address normalized solutions for the Dirac equation under such singular inhomogeneous conditions.

Motivated by these challenges, we establish the existence of normalized solutions to \eqref{1.1} for $b\in(0,1)$ and $p\in(2,3-b)$. The restriction $p<3-b$ is intrinsic in this setting, as it reflects the regularity loss caused by the singular weight \cite{Cotsiolis-2002} and defines the Sobolev-subcritical range for the problem. Our result covers the entire admissible range $p\in(2,3-b)$ in a unified way, effectively removing the traditional barriers between mass-subcritical, mass-critical and mass-supercritical regimes.

Our main result reads as follows.
\begin{theorem}\label{Th1}
Let $b\in(0,1)$ and $p\in(2,3-b)$. Then there exists a constant $a^*>0$ such that for every $a\in(0,a^*)$, there exist $\mu\in(0,m)$ and $u\in H^{\frac12}(\mathbb{R}^3;\mathbb{C}^4)$ for which $(\mu,u)$ is a normalized solution of \eqref{1.1}.
\end{theorem}

The proof of Theorem~\ref{Th1} builds on the variational framework of \cite{Chen-2024,CotiZelati-2025} (see also \cite{Nolasco-2021}). To accommodate the singular weight
$|x|^{-b}$, we establish a compact embedding into the weighted space $L^p(\mathbb{R}^3,|x|^{-b}\,\mathrm{d}x;\mathbb{C}^4)$, following the idea of \cite{Saanouni-2024}.
A key distinction between our strategy and that of \cite{Chen-2024} lies in how compactness is recovered.
In that work, compactness relies on taking a non-relativistic limit $c\to +\infty$, which widens the spectral gap and suppresses the negative energy component. Here we work with $c$ fixed (the full relativistic setting) and instead exploit the smallness of the prescribed mass $a$. This effectively weakens the nonlinear coupling enough to control the interaction between the positive and negative spectral subspaces without resorting the large-$c$ limit. Furthermore, the singular weight's dual nature is handled locally: the weighted embedding (Lemma~\ref{Lemma2.1}) absorbs the singularity at the origin, while the decay of $|x|^{-b}$ at infinity provides the compactness needed at large distances. In addition, unlike \cite{Chen-2024, CotiZelati-2025} which focus on the mass-subcritical nonlinearities, our framework yields a unified existence theory spanning the mass-subcritical to mass-supercritical range.

The paper is organized as follows. Section \ref{Preliminary results} introduces the functional-analytic framework, reviews the spectral decomposition of the free Dirac operator, and establishes a compact embedding into the weighted space $L^p(\mathbb{R}^3,|x|^{-b}dx;\mathbb{C}^4)$. Section \ref{existence} is devoted to the min-max reduction and the analysis of the reduced functional on the positive spectral subspace; together these yield the proof of Theorem \ref{Th1}.

{\bf Notations.}
\begin{itemize}
    \item For $R>0$, $B_R$ stands for the open ball of radius $R$ centered at the origin.
    \item The symbol $\|\cdot\|_{L^p}$ denotes the standard Lebesgue norm on $L^p(\mathbb{R}^3,\mathbb{C}^4)$ or $L^p(\mathbb{R}^3,\mathbb{C}^2)$.
    \item The expressions $\alpha\cdot\nabla$ and $\sigma\cdot\nabla$ abbreviate the sums $\sum\limits_{k=1}^3\alpha_k\partial_k$ and $\sum\limits_{k=1}^3\sigma_k\partial_k$, respectively.
    \item For a complex-valued expression, $\Re$ extracts its real part.
    \item $u\cdot v$ denotes the scalar product in $\C^4$ of $u$ and $v$, i.e., $u\cdot v=\sum\limits_{i=1}^4u_i\bar{v_i}$.
\end{itemize}

\section{Preliminary results}\label{Preliminary results}

We denote the Fourier transform of $u$ by $\mathcal{F}(u)$ or $\hat{u}$, defined by
\[
\hat{u}(\xi) = \frac{1}{(2\pi)^{\frac32}}\int_{\mathbb{R}^3} e^{-i\xi\cdot x}u(x) dx.
\]
The fractional Sobolev space $H^{\frac12}(\mathbb{R}^3,\mathbb{C}^4)$ is equipped with the inner product
\[
(u,v)_{H^{\frac12}} := \Re \int_{\mathbb{R}^3} \sqrt{|\xi|^2+1} \bigl(\hat{u}(\xi),\hat{v}(\xi)\bigr)_{\mathbb{C}^4} d\xi.
\]
We consider the free Dirac operator
\[
H_0 := -i\sum_{k=1}^{3}\alpha_k\partial_k + m\beta,
\]
which is self-adjoint on $L^2(\mathbb{R}^3,\mathbb{C}^4)$ with domain $\mathcal{D}(H_0) = H^1(\mathbb{R}^3,\mathbb{C}^4)$. Its spectrum is purely continuous:
\[
\sigma(H_0) = (-\infty, -m] \cup [m, +\infty).
\]
Consequently, the Hilbert space $L^2(\mathbb{R}^3,\mathbb{C}^4)$ admits the orthogonal decomposition
\[
L^2(\mathbb{R}^3,\mathbb{C}^4) = L^- \oplus L^+,
\]
where $H_0$ is positive definite on $L^+$ and negative definite on $L^-$.

Define the energy space $E$ as the completion of $\mathcal{D}(|H_0|^{\frac12})$ with respect to the inner product
\[
(u_1,u_2)_E := \bigl(|H_0|^{\frac12}u_1, |H_0|^{\frac12}u_2\bigr)_{L^2}.
\]
The corresponding norm is $\|u\|_E := (u,u)_E^{\frac12}$.  It is straightforward to verify that $E$ coincides with $H^{\frac12}(\mathbb{R}^3,\mathbb{C}^4)$, and there holds
\[
m\|u\|_{L^2}^2 \le \|u\|_E^2, \quad
\min\{m,1\} \|u\|_{H^{\frac12}}^2 \le \|u\|_E^2 \le \max\{m,1\} \|u\|_{H^{\frac12}}^2.
\]

The space $E$ inherits the orthogonal decomposition
\[
E = E^- \oplus E^+, \quad \text{where~~} E^\pm := E \cap L^\pm.
\]
Let $P^\pm := \frac{1}{2}\bigl(I \pm \frac{H_0}{|H_0|}\bigr)$ denote the orthogonal projections onto $E^\pm$. Their Fourier representations are essential for the subsequent analysis. In the Fourier domain, the symbol of $H_0$ is the Hermitian $4\times 4$ matrix
\[
\widehat{H}_0(\xi) = \begin{pmatrix}
m I_2 & \sigma\cdot\xi \\[2pt]
\sigma\cdot\xi & -m I_2
\end{pmatrix},
\]
whose eigenvalues are $\pm\lambda(\xi)$ with $\lambda(\xi) = \sqrt{m^2 + |\xi|^2}$. A unitary matrix that diagonalizes $\widehat{H}_0(\xi)$ is given explicitly by
\[
\mathbf{U}(\xi) = \frac{\bigl(m + \lambda(\xi)\bigr)I_4 + \beta\alpha\cdot\xi}{\sqrt{2\lambda(\xi)\bigl(m + \lambda(\xi)\bigr)}}
= \rho_+ I_4 + \rho_- \beta \frac{\alpha\cdot\xi}{|\xi|},
\]
where $\rho_\pm := \sqrt{\frac{1}{2}\bigl(1 \pm \frac{m}{\lambda(\xi)}\bigr)}$. Its inverse takes the form
\[
\mathbf{U}^{-1}(\xi) = \frac{\bigl(m + \lambda(\xi)\bigr)I_4 - \beta\alpha\cdot\xi}{\sqrt{2\lambda(\xi)\bigl(m + \lambda(\xi)\bigr)}}
= \rho_+ I_4 - \rho_- \beta \frac{\alpha\cdot\xi}{|\xi|},
\]
and one verifies the diagonalization
\[
\mathbf{U}(\xi)\widehat{H}_0(\xi)\mathbf{U}^{-1}(\xi) = \lambda(\xi)\beta.
\]
Consequently, the Fourier representation of the projections $P^\pm$ reads
\[
\widehat{P^\pm u}(\xi) = \frac{1}{2}\mathbf{U}^{-1}(\xi)(I_4 \pm \beta)\mathbf{U}(\xi)\hat{u}(\xi)
= \frac{1}{2}\Bigl(I_4 \pm \frac{m}{\lambda(\xi)}\beta \pm \frac{1}{\lambda(\xi)}\alpha\cdot\xi\Bigr)\hat{u}(\xi).
\]
Finally, the Foldy-Wouthuysen transformation $\mathbf{U}_{\mathrm{FW}} := \mathcal{F}^{-1}\mathbf{U}\mathcal{F}$ block-diagonalizes the free Dirac operator:
\[
\mathbf{U}_{\mathrm{FW}}\,H_0\,\mathbf{U}_{\mathrm{FW}}^{-1}
= \begin{pmatrix}
\sqrt{-\Delta + m^2}\,I_2 & 0 \\[4pt]
0 & -\sqrt{-\Delta + m^2}\,I_2
\end{pmatrix}
= \beta|H_0|.
\]
This transformation serves as a key tool in the subsequent reduction arguments.

%\begin{lemma}
%Let $0<b<1$ and $2<q<3-b$. Then there exists a constant $C>0$ such that for every $u\in H^{\frac12}(\mathbb{R}^3)$,
%\[
%\int_{\mathbb{R}^3}|x|^{-b}|u|^q\,dx \;\le\; C\; \|(-\Delta)^{\frac{1}{4}}u\|_2^{\,3(q-2)+2b}\; \|u\|_2^{\,6-2q-2b}.
%\]
%\end{lemma}
%
%\begin{proof}
%We start from the sharp inequality of [Saanouni and Shi  Theorem 2.1]. With $N=3$, $\gamma=1/2$, $\rho=-b$ and $p=q-1$, it yields
%\[
%\int_{\mathbb{R}^3}|x|^{-b}|u|^q\,dx \;\le\; \frac{2}{q}\|\phi\|_{p_c}^{\,q-2}\; \|u\|_{p_c}^{\,q-2}\; \|(-\Delta)^{\frac{1}{4}}u\|_2^2,
%\]
%where $p_c = \dfrac{3(q-2)}{1-b}$ and $\phi$ is a ground state solution of
%\[
%(-\Delta)^{\frac{1}{2}}\phi + |\phi|^{p_c-2}\phi = |x|^{-b}|\phi|^{q-2}\phi \quad\text{in }\mathbb{R}^3.
%\]
%
%To eliminate the $\|u\|_{p_c}$ factor, we employ the fractional Sobolev inequality [Ding, CV, Lemma2.3] with $N=3$, $p=2$, $s=\frac{1}{2}$, $p_1=p_c$):
%\[
%\|u\|_{p_c} \;\le\; C_{\mathrm{Sob}}\; \|(-\Delta)^{\frac{1}{4}}u\|_2^{\theta}\; \|u\|_2^{1-\theta},
%\]
%with $\theta = 3- \frac{6}{p_c}$. Substituting into the previous estimate gives
%\[
%\int_{\mathbb{R}^3}|x|^{-b}|u|^q\,dx \;\le\; \frac{2}{q}\|\phi\|_{p_c}^{\,q-2} C_{\mathrm{Sob}}^{\,q-2}\; \|(-\Delta)^{\frac{1}{4}}u\|_2^{2+(q-2)\theta}\; \|u\|_2^{(q-2)(1-\theta)}.
%\]
%A direct computation shows
%\[
%2+(q-2)\theta = 3(q-2)+2b,\qquad (q-2)(1-\theta)=6-2q-2b.
%\]
%Thus the inequality holds with $C = \frac{2}{q}\|\phi\|_{p_c}^{\,q-2} C_{\mathrm{Sob}}^{\,q-2}$, which depends only on $q$ and $b$.
%\end{proof}

The following embedding restores the compactness lost due to the singular weight $|x|^{-b}$.
\begin{lemma}\label{Lemma2.1}
Let $0<b<1$ and $2<p<3-b$. Then the embedding
\[
H^{\frac12}(\mathbb{R}^3;\C^4)\hookrightarrow L^p(\mathbb{R}^3,|x|^{-b}dx;\C^4)
\]
is compact.
\end{lemma}

\begin{proof}
We first prove the continuity of the embedding. For any $u \in H^{\frac{1}{2}}(\R^3;\C^4)$ and $p \in (2,3-b)$, applying the H\"{o}lder inequality together with the fractional Sobolev inequality $\|u\|_p \le \mathcal{S}_p \|u\|_{H^{\frac{1}{2}}}$ yields
\begin{align*}
\int_{\mathbb{R}^3}|x|^{-b}|u|^p dx
&= \int_{|x|\le1}|x|^{-b}|u|^p dx + \int_{|x|>1}|x|^{-b}|u|^p dx \\
&\le \Bigl(\int_{|x|\le1}|x|^{-\frac{3b}{3-p}} dx\Bigr)^{\frac{3-p}{3}}\Bigl(\int_{|x|\le1}|u|^3 dx\Bigr)^{\frac{p}{3}} + \mathcal{S}_p^p \|u\|_{H^{\frac{1}{2}}}^p \\
&\le \Bigl(\int_{|x|\le1}|x|^{-\frac{3b}{3-p}} dx\Bigr)^{\frac{3-p}{3}}\mathcal{S}_3^p \|u\|_{H^{\frac{1}{2}}}^p + \mathcal{S}_p^p \|u\|_{H^{\frac{1}{2}}}^p.
\end{align*}
Since $b \in (0,1)$ and $p \in (2,3-b)$, we have $\frac{3b}{3-p} < 3$, which guarantees the finiteness of the integral over the unit ball.  Denoting this bound by $C_1$, we obtain the global estimate
\begin{align*}
\|u\|_{L^p(|x|^{-b})}^p \le \mathcal{S}\|u\|_{H^{\frac{1}{2}}}^p~~\mbox{with}~~ \mathcal{S} = C_1\mathcal{S}_3^p + \mathcal{S}_p^p,
\end{align*}
 which completes the proof of continuity.

To prove compactness, let $\{u_n\} \subset H^{\frac{1}{2}}(\mathbb{R}^3;\C^4)$ be a bounded sequence. By the Rellich-Kondrachov theorem, the embedding $H^{\frac{1}{2}}(B_R) \hookrightarrow L^p(B_R)$ is compact for any $R>0$. A standard diagonal argument yields a subsequence, still denoted by $\{u_n\}$, that converges to some $u$ in $L^p(B_R)$ for all $R>0$. Set $v_n := u_n - u$. Our goal is to show that $v_n \to 0$ strongly in $L^p(\R^3, |x|^{-b}dx;\C^4)$. Since $\{u_n\}$ is bounded in $H^{\frac{1}{2}}$, we have $\|v_n\|_{H^{\frac{1}{2}}} \le 2M$ for some $M>0$ and $\|v_n\|_{L^p(B_R)} \to 0$ for each $R>0$. Furthermore, the fractional Sobolev inequality gives the uniform bound $\|v_n\|_p \le 2\mathcal{S}_p M$.

Fix an arbitrary $\varepsilon > 0$. We first choose $R_0$ large enough so that $(2\mathcal{S}_p M)^p R_0^{-b} < \frac{\varepsilon}{3}$, and then select $\delta = \frac12 \min\{1, R_0\}$. For every $n$, we split the weighted integral into three regions:
\begin{equation}\label{2.1}
\int_{\mathbb{R}^3} |x|^{-b} |v_n|^p  dx = \int_{|x|>R_0} |x|^{-b} |v_n|^p  dx + \int_{B_{R_0}\setminus B_\delta} |x|^{-b} |v_n|^p dx + \int_{B_\delta} |x|^{-b} |v_n|^p dx.
\end{equation}
The contribution outside the large ball is controlled directly by the choice of $R_0$:
\[
\int_{|x|>R_0} |x|^{-b} |v_n|^p dx \le R_0^{-b} \|v_n\|_p^p \le R_0^{-b} (2\mathcal{S}_p M)^p < \frac{\varepsilon}{3}.
\]
For the annular region $B_{R_0} \setminus B_\delta$, we use the bound $|x|^{-b} \le \delta^{-b}$. Since $v_n \to 0$ in $L^p(B_{R_0})$, there exists an integer $N_1$ such that for all $n \ge N_1$,
\[
\int_{B_{R_0}\setminus B_\delta} |x|^{-b} |v_n|^p dx \le \delta^{-b} \|v_n\|_{L^p(B_{R_0})}^p < \frac{\varepsilon}{3}.
\]
It remains to estimate the integral over the small ball $B_\delta$. Observe that the condition $p < 3-b$ is equivalent to $\frac{3}{3-p} < \frac{3}{b}$. Hence we can select an auxiliary exponent $\theta$ satisfying $\frac{3}{3-p} < \theta < \frac{3}{b}$. Let $\theta' = \frac{\theta}{\theta-1}$ be its H\"older conjugate. The choice of $\theta$ guarantees that $b\theta < 3$ and $2<p\theta' < 3$. Applying H\"older's inequality on $B_\delta$ yields
\[
\int_{B_\delta} |x|^{-b} |v_n|^p dx \le \Bigl(\int_{B_\delta} |x|^{-b\theta} dx\Bigr)^{\frac{1}{\theta}} \Bigl(\int_{B_\delta} |v_n|^{p\theta'} dx\Bigr)^{\frac{1}{\theta'}}.
\]
Define $C_\delta := \bigl(\int_{B_\delta} |x|^{-b\theta} dx\bigr)^{\frac{1}{\theta}}$. Because $b\theta < 3$, we have $C_\delta \to 0$ as $\delta \to 0$. Since $p\theta' \in (2,3)$, the compact embedding $H^{\frac{1}{2}}(B_\delta) \hookrightarrow L^{p\theta'}(B_\delta)$ implies that $\|v_n\|_{L^{p\theta'}(B_\delta)} \to 0$. Thus, for the fixed $\delta$ chosen above, there exists $N_2$ such that for all $n \ge N_2$,
\[
\Bigl(\int_{B_\delta} |v_n|^{p\theta'} dx\Bigr)^{\frac{1}{\theta'}} < \frac{\varepsilon}{3 C_\delta},
\]
which directly gives $\int_{B_\delta} |x|^{-b} |v_n|^p dx < \frac{\varepsilon}{3}$.

Collecting the estimates for the three regions in \eqref{2.1}, we conclude that for all $n \ge \max\{N_1, N_2\}$,
\[
\int_{\mathbb{R}^3} |x|^{-b} |v_n|^p dx < \varepsilon.
\]
Since $\varepsilon > 0$ is arbitrary, this proves $\|v_n\|_{L^p(|x|^{-b})} \to 0$. Consequently, a subsequence of $\{u_n\}$ converges strongly in $L^p(\R^3, |x|^{-b}dx;\C^4)$, thereby establishing the compactness of the embedding.
\end{proof}

\section{Proof of Theorem \ref{Th1}}\label{existence}
Throughout this section, we identify $E$ with $H^{\frac{1}{2}}(\R^3;\C^4)$ and recall the norm equivalence
$$
\mathcal{C}_1\|u\|_{H^{\frac{1}{2}}} \le \|u\|_E \le \mathcal{C}_2\|u\|_{H^{\frac{1}{2}}}, \quad \forall u \in E.
$$
We introduce the energy functional $I : E \to \R$ associated with \eqref{1.1}:
\begin{equation*}
I(u) = \|u^+\|_E^2 - \|u^-\|_E^2 - \frac{2}{p} \int_{\R^3} |x|^{-b} |u|^p dx.
\end{equation*}
It is clear that (weak) normalized solutions of \eqref{1.1} correspond to critical points of $I$ restricted to the $L^2$-sphere
\[
S_a := \bigl\{ u \in E : \|u\|_{L^2}^2 = a \bigr\}.
\]
Set
\[
\chi := 2(m+1)a,\quad
\mathcal{V}_\chi := \{u\in E : \|u\|_E^2 < \chi\},\quad
\mathcal{V}_\chi^+ := \mathcal{V}_\chi\cap E^+.
\]
We now employ a reduction method to study the normalized solutions of \eqref{1.1}.

\subsection{Maximization problem}

For any $v \in S_a\cap\mathcal{V}_\chi^+$ with $\|v\|_{L^2}^2 = a$, let $W = \operatorname{span}\{v\}$ and consider
\[
S_{a,W} := \{ u \in S_a : u^+ \in W \}.
\]
We aim to maximize the functional $I$ on $S_{a,W}$. The tangent space at  $u \in S_{a,W}$ is given by
\[
T_u(S_{a,W}) = \{ h \in W \oplus E^- : \Re(u,h)_{L^2} = 0 \}.
\]
We denote by $dI|_{S_{a,W}}(u)$ the projection of $dI(u)$ onto $T_u(S_{a,W})$, defined by
\[
dI|_{S_{a,W}}(u)[h] = dI(u)[h] - 2\mu(u)\Re(u,h)_{L^2}, \quad \forall  h \in W \oplus E^-,
\]
where $\mu(u) \in \R$ is chosen so that $dI|_{S_{a,W}}(u) \in T_u(S_{a,W})$.

For a fixed $v \in S_a\cap\mathcal{V}_\chi^+$, we say that a sequence $\{u_n\} \subset S_{a,W}$ is a Palais-Smale sequence for $I$ on $S_{a,W}$ at level $e_a$ if
\[
I(u_n) \to e_a \quad \text{and} \quad \|dI|_{S_{a,W}}(u_n)\| \to 0, \quad \text{as } n \to \infty.
\]

Set
$$
a_1:=\left(\frac{m\mathcal{C}_1^p }{2^{\frac{3p+6}{2}}\mathcal{S}(m+1)^{\frac p2}}\right)^{\frac{2}{p-2}}.
$$
\begin{lemma}\label{Lemma3.1}
Fix $a \in (0, a_1)$, and let $\{u_n\} \subset S_{a,W}$ be a Palais-Smale sequence for the functional
\[
I(u)=\|u^+\|_E^2-\|u^-\|_E^2-\frac{2}{p}\int_{\mathbb{R}^3}|x|^{-b}|u|^p dx
\]
on $S_{a,W}$, i.e.,
\[
I(u_n)\to\ell,\quad \|dI|_{S_{a,W}}(u_n)\|\to0,
\]
with $\ell\ge\frac{ma}{2}$. Then\\
(i) $\{u_n\}$ is bounded in $E$;\\
(ii) $\liminf\limits_{n\to\infty}\mu(u_n)>0$;\\
(iii) $\{u_n\}$ is precompact in $H^{\frac{1}{2}}(\R^3;\C^4)$.
\end{lemma}

\begin{proof}
Write $u_n=t_nv+u_n^-$ with $v\in S_a\cap\mathcal{V}_\chi^+$ fixed and $u_n^-\in E^-$. From $\|u_n\|_{L^2}^2=a$ and $\|v\|_{L^2}^2=a$, we obtain $|t_n|^2a+\|u_n^-\|_{L^2}^2=a$, hence $|t_n|\le1$. Then we obtain that
\[
I(u_n)=\|u_n^+\|_E^2-\|u_n^-\|_E^2-\frac{2}{p}\int|x|^{-b}|u_n|^pdx
\le t_n^2\|v\|_E^2-\|u_n^-\|_E^2\le\|v\|_E^2-\|u_n^-\|_E^2.
\]
Thus $\|u_n^-\|_E^2\le\|v\|_E^2-I(u_n)$. Since $\{I(u_n)\}$ is bounded and $\{u_n^-\}$ is bounded in $E$, together with $|t_n|\le1$, we conclude that $\{u_n\}$ is bounded in $E$. This proves (i).

It follows from $\{u_n\} \subset S_{a,W}$ is a Palais-Smale sequence for $I$ on $S_{a,W}$ that there exists $\mu(u_n)\in\mathbb{R}$ such that for every $h\in W\oplus E^-$,
\[
dI(u_n)[h]=2\mu(u_n)\Re(u_n, h)_{L^2}+o_n(1).
\]
Taking $h=u_n$ (note $u_n\in W\oplus E^-$) gives
\begin{equation}\label{3.1}
\|u_n^+\|_E^2-\|u_n^-\|_E^2-\int_{\mathbb{R}^3}|x|^{-b}|u_n|^p dx=\mu(u_n)a+o_n(1).
\end{equation}
On the other hand,
\[
I(u_n)=\|u_n^+\|_E^2-\|u_n^-\|_E^2-\frac{2}{p}\int_{\mathbb{R}^3}|x|^{-b}|u_n|^p dx\to\ell.
\]
By using \eqref{3.1}, we obtain
\[
I(u_n)=\mu(u_n)a+\Bigl(1-\frac{2}{p}\Bigr)\int_{\mathbb{R}^3}|x|^{-b}|u_n|^p dx+o_n(1).
\]
Hence
\begin{align*}
\mu(u_n)a &= I(u_n)-\Bigl(1-\frac{2}{p}\Bigr)\int_{\mathbb{R}^3}|x|^{-b}|u_n|^p dx+o(1)\\
&\ge \frac{\ell}{2}-\Bigl(1-\frac{2}{p}\Bigr)\int_{\mathbb{R}^3}|x|^{-b}|u_n|^p dx+o_n(1).
\end{align*}
By the weighted Sobolev embedding $H^{\frac{1}{2}}(\mathbb{R}^3;\C^4)\hookrightarrow L^p(\mathbb{R}^3,|x|^{-b}dx;\C^4)$ and the fact  $\|u_n^-\|_E^2\leq \|u_n^+\|_E^2$, we have
$$\Bigl(1-\frac{2}{p}\Bigr)\int|x|^{-b}|u_n|^p dx\le \frac{\mathcal{S}}{\mathcal{C}_1^p}\|u_n\|_E^p\le \frac{2^{p}\mathcal{S}[2(m+1)]^{\frac{p}{2}}}{\mathcal{C}_1^p}a^{\frac{p}{2}}.$$
Consequently,
\[
\mu(u_n)a\ge\frac{\ell}{2}-\frac{2^{\frac{3p}{2}}\mathcal{S}(m+1)^{\frac{p}{2}}}{\mathcal{C}_1^p }a^{\frac p2}\ge\frac{ma}{4}-\frac{2^{\frac{3p}{2}}\mathcal{S}(m+1)^{\frac{p}{2}}}{\mathcal{C}_1^p }a^{\frac p2}>\frac{ma}{8}.
\]
Hence $\liminf\limits_{n\to\infty}\mu(u_n)>0$. This proves (ii).

To prove (iii), extract a subsequence (still denoted $\{u_n\}$) such that $u_n \rightharpoonup u$ in $E$. Since $\dim W = 1$, we have $u_n^+ \to u^+$ in $E$. Set $\varphi_n = u_n^- - u^- \in E^-$. Since $\{u_n\} \subset S_{a,W}$ be a Palais-Smale sequence for $I$ on $S_{a,W}$, then we have
\begin{equation}\label{3.2}
o(1) = -\frac12 dI(u_n)[\varphi_n] + \mu(u_n) \|\varphi_n\|_{L^2}^2.
\end{equation}
Expanding the derivative yields
\[
dI(u_n)[\varphi_n] = 2(u_n^+, 0)_E - 2(u_n^-, \varphi_n)_E - 2\int_{\mathbb{R}^3} |x|^{-b} |u_n|^{p-2} \Re(u_n, \varphi_n) dx.
\]
Since $(u_n^+, 0)_E = 0$ and $-2(u_n^-, \varphi_n)_E = -2\|\varphi_n\|_E^2 - 2(u^-, \varphi_n)_E$, substituting into \eqref{3.2} gives
\[
o(1) = \|\varphi_n\|_E^2 + (u^-, \varphi_n)_E + \int_{\mathbb{R}^3} |x|^{-b} |u_n|^{p-2} \Re(u_n, \varphi_n) dx + \mu(u_n) \|\varphi_n\|_{L^2}^2.
\]
Because $\varphi_n \rightharpoonup 0$ in $E$, we have $(u^-, \varphi_n)_E \to 0$, and $\mu(u_n)$ stays bounded away from zero. It remains to show that
\[
\int_{\mathbb{R}^3} |x|^{-b} |u_n|^{p-2} \Re(u_n ,\varphi_n) dx \to 0.
\]
By the pointwise bound $|\Re(u_n ,\varphi_n)| \le |u_n| |\varphi_n|$ and H\"older's inequality,
\[
\Bigl| \int_{\mathbb{R}^3} |x|^{-b} |u_n|^{p-2} \Re(u_n,\varphi_n) dx \Bigr|
\le \|u_n\|_{L^p(|x|^{-b})}^{p-1} \|\varphi_n\|_{L^p(|x|^{-b})}.
\]
Since the embedding $H^{\frac{1}{2}}(\R^3;\C^4) \hookrightarrow L^p(\R^3,|x|^{-b}dx;\C^4)$ is compact for $p\in(2,3-b)$, the sequence $\{u_n\}$ is bounded in $L^p(\R^3,|x|^{-b}dx;\C^4)$, and the weak convergence $\varphi_n \rightharpoonup 0$ in $H^{\frac{1}{2}}(\R^3;\C^4)$ implies $\|\varphi_n\|_{L^p(|x|^{-b})} \to 0$. Thus the integral tends to zero, yielding $\|\varphi_n\|_E \to 0$. Combined with $u_n^+ \to u^+$ in $H^{\frac{1}{2}}(\mathbb{R}^3; \mathbb{C}^4)$, this gives $u_n \to u$ strongly in $H^{\frac{1}{2}}(\mathbb{R}^3; \mathbb{C}^4)$.
\end{proof}

In order to establish the uniqueness of the maximizer, we compute the second variation of $I$ at a critical point on $S_{a,W}$.
\begin{lemma}\label{Lemma3.2}
Let $a\in(0,a_1)$, $v\in S_a\cap\mathcal{V}_\chi^+$ and $W=span\{v\}$. Suppose $u\in E$ satisfies
\[
dI(u)[h] - 2\mu(u)\Re(u,h)_{L^2} = 0 \quad \text{for all } h\in W\oplus E^-,
\]
and $I(u) \ge \frac{ma}{2}$. Then for every $h\in T_uS_{a,W}$, we have
\[
d^2I(u)[h,h] - 2\mu(u)\|h\|_{L^2}^2< 0.
\]
Whence, $u$ is a strict local maximum of $I$ on the constrained manifold $S_{a,W}$.
\end{lemma}
\begin{proof}
Write $u=\tau v+\eta$ with $\eta\in E^-$. From $\|u\|_{L^2}^2=a$ we have $\tau^2 a+\|\eta\|_{L^2}^2=a$, hence $\tau=\sqrt{1-\frac{1}{a}\|\eta\|_{L^2}^2}$.

For any $h\in T_uS_{a,W}$, decompose $h=cv+\xi$ with $\xi\in E^-$. The orthogonality condition $\Re(u,h)_{L^2}=0$ gives
\[
c=-\frac{1}{\tau a}\Re(\eta,\xi)_{L^2}.
\]
Using bilinearity,
\begin{equation}\label{3.3}
\begin{aligned}
d^2I(u)[h,h]&=\tau^{-1}c d^2I(u)[u,cv-\tau^{-1}c\eta] +2d^2I(u)[cv,\xi] \\
&\quad +\tau^{-2}c^2 d^2I(u)[\eta,\eta]+d^2I(u)[\xi,\xi].
\end{aligned}
\end{equation}
Recall that
\[
dI(u)[h]=2(u^+,h^+)_E-2(u^-,h^-)_E-2\int_{\R^3}|x|^{-b}|u|^{p-2}\Re(u,h) dx,
\]
and
\begin{align*}
d^2I(u)[v,h]&=2(v^+,h^+)_E-2(v^-,h^-)_E-2\int_{\R^3}|x|^{-b}|u|^{p-4}\Bigl[(p-2)\Re(u,h)\Re(u,v)\\
&\quad+|u|^2\Re(v,h)\Bigr]dx.
\end{align*}
By the H\"{o}lder's inequality, the estimate $\|u\|_{L^p(|x|^{-b})}^p \le \mathcal{S}\|u\|_{H^{\frac{1}{2}}}^p$ and the equivalence $E \cong H^{\frac{1}{2}}(\R^3;\C^4)$, we obtain that
\begin{equation}\label{3.4}
\begin{aligned}
&\tau^{-1}c d^2I(u)[u,cv-\tau^{-1}c\eta] \\
&=2\tau^{-1}c dI(u)[cv-\tau^{-1}c\eta]-2\tau^{-1}c(\tau v,cv)_E+2\tau^{-1}c(\eta,-\tau^{-1}c\eta)_E\\
&\quad-2\tau^{-1}c(p-3)\int_{\R^3}|x|^{-b}|u|^{p-2}\Re(u,cv-\tau^{-1}c\eta) dx\\
&=4\mu(u)c^2 a-4\mu(u)\tau^2c^2\|\eta\|_{L^2}^2-2c^2\|v\|_E^2-2\tau^{-2}c^2\|\eta\|_E^2\\
&\quad-2(p-3)c^2\int_{\R^3}|x|^{-b}|u|^{p-2}|v|^2dx+2(p-3)\tau^{-2}c^2\int_{\R^3}|x|^{-b}|u|^{p-2}|\eta|^2dx\\
&\le 2\mu(u)\|h\|_{L^2}^2-2c^2\|v\|_E^2+2(3-p)c^2\int_{\R^3}|x|^{-b}|u|^{p-2}|v|^2dx\\
&\le 2\mu(u)\|h\|_{L^2}^2-2c^2\|v\|_E^2+2(3-p)c^2
\left(\int_{\R^3}|x|^{-b}|u|^pdx\right)^{\frac{p-2}{p}}
\left(\int_{\R^3}|x|^{-b}|v|^pdx\right)^{\frac{2}{p}}\\
&\le 2\mu(u)\|h\|_{L^2}^2-2\|h^+\|_E^2+
\frac{2^{\frac{3p}{2}}\mathcal{S}(m+1)^{\frac{p-2}{2}}}
{\mathcal{C}_1^p}a^{\frac{p-2}{2}}\|h\|_E^2.
\end{aligned}
\end{equation}
Similarly, we have
\begin{equation}\label{3.5}
\begin{aligned}
|2d^2I(u)[cv,\xi]|&=\Bigl|-4\int_{\R^3}|x|^{-b}|u|^{p-2}\Re(cv,\xi)dx \\
&\quad-4(p-2)\int_{\R^3}|x|^{-b}|u|^{p-4}\Re(u,cv)\Re(u,\xi)dx\Bigr|\\
&\le \frac{2^{\frac{3p}{2}}\mathcal{S}(m+1)^{\frac{p-2}{2}}}
{\mathcal{C}_1^p}a^{\frac{p-2}{2}}\|h\|_E^2.
\end{aligned}
\end{equation}
We now estimate the term involving $\eta$. Since $\eta \in E^-$, it follows that
\begin{equation}\label{3.6}
\begin{aligned}
\tau^{-2}c^2 d^2I(u)[\eta,\eta]
&=-2\tau^{-2}c^2\|\eta\|_E^2-2\tau^{-2}c^2\int_{\R^3}|x|^{-b}|u|^{p-2}|\eta|^2dx\\
&\quad-2(p-2)\tau^{-2}c^2\int_{\R^3}|x|^{-b}|u|^{p-4}[\Re(u,\eta)]^2dx\\
&\le 0.
\end{aligned}
\end{equation}
Finally, for $\xi\in E^-$ we have that
\begin{align*}
d^2I(u)[\xi,\xi]
&=-2\|\xi\|_E^2-2\int_{\R^3}|x|^{-b}|u|^{p-2}|\xi|^2dx\\
&\quad-2(p-2)\int_{\R^3}|x|^{-b}|u|^{p-4}[\Re(u,\xi)]^2dx\\
&\le -2\|h^-\|_E^2.
\end{align*}
Inserting \eqref{3.4}--\eqref{3.6} and the estimate for $d^2I(u)[\xi,\xi]$ into \eqref{3.3}, we obtain
\[
d^2I(u)[h,h]-2\mu(u)\|h\|_{L^2}^2\le -2\bigl[1-\frac{2^{\frac{3p}{2}}\mathcal{S}(m+1)^{\frac{p-2}{2}}}
{\mathcal{C}_1^p}a^{\frac{p-2}{2}}\bigr]\|h\|_E^2<0.
\]
So $u$ is a strict local maximum.
\end{proof}

We now turn to the existence of a global maximizer. For any $v\in S_a\cap\mathcal{V}_\chi^+$, consider the maximization problem
\[
\Psi_W:=\sup_{u\in S_{a,W}}I(u).
\]
\begin{lemma}\label{Lemma3.3}
Let $a\in(0,a_1)$. Then for any $v\in S_a\cap\mathcal{V}_\chi^+$, we have
\[
\frac{ma}{2} < \Psi_W \le \|v\|_E^2.
\]
\end{lemma}
\begin{proof}
The upper bound is immediate: for any $u\in S_{a,W}$,
\[
I(u)=\|u^+\|_E^2-\|u^-\|_E^2-\frac{2}{p}\int_{\R^3}|x|^{-b}|u|^p dx \le \|u^+\|_E^2 \le \|v\|_E^2.
\]
For the lower bound, note that $E\cong H^{\frac{1}{2}}(\R^3;\C^4)$ and by the weighted Sobolev embedding $H^{\frac{1}{2}}(\R^3;\C^4)\hookrightarrow L^p(\R^3,|x|^{-b}dx;\C^4)$ (valid because $p\in(2,3-b)$) yield that
\[
\int|x|^{-b}|v|^p dx\le \frac{\mathcal{S}}{\mathcal{C}_1^p}\|v\|_E^p< \frac{[2(m+1)]^{\frac{p-2}{2}}\mathcal{S}}{\mathcal{C}_1^p} a^{\frac{p-2}{2}}\|v\|_E^2.
\]
In addition, we have $\|v\|_E^2\ge m\|v\|_{L^2}^2=ma$. Hence
\begin{align*}
\Psi_W &\ge I(v)=\|v\|_E^2-\frac{2}{p}\int_{\R^3}|x|^{-b}|v|^p dx \ge \left(1-\frac{2^{\frac p2}(m+1)^{\frac{p-2}{2}}\mathcal{S}}
{p\mathcal{C}_1^p}a^{\frac{p-2}{2}}\right)\|v\|_E^2\\
&>\frac{ma}{2}.
\end{align*}
This completes the proof.
\end{proof}

%Combining the previous lemmas, we can now prove the existence and smoothness of the maximizer.
Set
\[
a^*:=\min\left\{a_1,
\left(\frac{p\mathcal{C}_1^p}
{(\frac{3m+2}{m}+2^{p})2^{\frac{p+4}{2}}(m+1)^{\frac{p}{2}}\mathcal{S}}\right)^{\frac{2}{p-2}},
\left(\frac{pm^2\mathcal{C}_1^p}{2^{\frac{p+10}{2}}(m+1)^{\frac {p+2}{2}}\mathcal{S}}\right)^{\frac{2}{p-2}}\right\}.
\]
\begin{lemma}\label{Lemma3.4}
Let $a\in(0,a^*)$. Then for any $v \in S_a\cap\mathcal{V}_\chi^+$, there exists a unique element $\phi(v) \in S_{a,W}$ which is the strict global maximizer of $I$ on $S_{a,W}$, that is,
\[
I(\phi(v)) = \sup_{u \in S_{a,W}} I(u) = \Psi_W.
\]
Moreover, the mapping $\phi \colon S_a\cap\mathcal{V}_\chi^+ \to S_{a,W}$, $v \mapsto \phi(v)$, is $C^1$, and for every $v \in S_a\cap\mathcal{V}_\chi^+$,
\[
dI(\phi(v))[h] - 2\mu(\phi(v)) \operatorname{Re}(\phi(v), h)_{L^2} = 0, \quad \forall h \in W \oplus E^-.
\]
\end{lemma}
\begin{proof}
Since $\Psi_W > \frac{ma}{2}$ and $I$ is bounded above on $S_{a,W}$, Ekeland's variational principle yields a Palais-Smale sequence $\{u_n\} \subset S_{a,W}$ for $I$ at level $\Psi_W$. By Lemma \ref{Lemma3.1}, this sequence is precompact in $E$ (because $\Psi_W \ge \frac{ma}{2}$ for sufficiently small $a$), so a subsequence converges to some $\phi(v)\in S_{a,W}$. Then
\[
I(\phi(v)) = \Psi_W = \sup_{u\in S_{a,W}} I(u), \quad \|dI|_{S_{a,W}}(\phi(v))\| = 0.
\]

We first prove uniqueness. Suppose there exist two distinct maximizers $\phi_1,\phi_2\in S_{a,W}$ with
\[
I(\phi_1)=I(\phi_2)=\sup_{u\in S_{a,W}}I(u)=\Psi_W>\frac{ma}{2}.
\]
Write $\phi_i=\tau_i v+\eta_i$ with $\eta_i\in E^-$ and $\tau_i=\sqrt{1-\frac{1}{a}\|\eta_i\|_{L^2}^2}$.
Since $I(v)\le I(\phi_i)$ and $I(\phi_i)\le\|v\|_E^2-\|\eta_i\|_E^2$, we obtain
\begin{equation}\label{3.7}
\|\eta_i\|_E^2\le\|v\|_E^2-I(v)=\frac{2}{p}\int_{\mathbb{R}^3}|x|^{-b}|v|^p dx\le \frac{2^{\frac p2}(m+1)^{\frac{p-2}{2}}\mathcal{S}}{p\mathcal{C}_1^p} a^{\frac{p-2}{2}}\|v\|_E^2,
\end{equation}
where the last inequality follows from $\|u\|_{L^p(|x|^{-b})}^p \le \mathcal{S}\|u\|_{H^{\frac{1}{2}}}^p$ and the equivalence $E\cong H^{\frac12}(\mathbb{R}^3;\mathbb{C}^4)$.

Define the set of paths
\[
\Gamma:=\{\gamma:[0,1]\to S_{a,W}\mid \gamma(0)=\phi_1,\ \gamma(1)=\phi_2\}
\]
and the minimax level
\[
l:=\sup_{\gamma\in\Gamma}\min_{r\in[0,1]}I(\gamma(r)).
\]

Choose the linear path
\[
\gamma(r)=\tau(r)v+\eta(r),\quad \eta(r)=r\eta_2+(1-r)\eta_1,\quad \tau(r)=\sqrt{1-\frac{1}{a}\|\eta(r)\|_{L^2}^2},
\]
which clearly belongs to $S_{a,W}$. Then we have
\begin{equation}\label{3.8}
\int|x|^{-b}|\gamma(r)|^p dx\le \frac{\mathcal{S}}{\mathcal{C}_1^p}\|\gamma(r)\|_E^p\le \frac{2^{\frac{3p-2}{2}}(m+1)^{\frac{p-2}{2}}\mathcal{S}}
{\mathcal{C}_1^p}a^{\frac{p-2}{2}}\|v\|_E^2.
\end{equation}
It follows from \eqref{3.7}, \eqref{3.8} and the fact that $\|\eta(r)\|_{L^2}^2\le\frac{1}{m}\|\eta(r)\|_E^2$ that
\[
\begin{aligned}
I(\gamma(r))&=\|\tau(r)v\|_E^2-\|\eta(r)\|_E^2-\frac{2}{p}\int_{\mathbb{R}^3}|x|^{-b}|\gamma(r)|^p dx\\
&=\left(1-\frac{1}{a}\|\eta(r)\|_{L^2}^2\right)\|v\|_E^2
-\|\eta(r)\|_E^2-\frac{2}{p}\int_{\mathbb{R}^3}|x|^{-b}|\gamma(r)|^p dx\\
&\ge \left[1-\frac{(\frac{3m+2}{m}+2^{p})2^{\frac p2}(m+1)^{\frac{p-2}{2}}\mathcal{S}}
{p\mathcal{C}_1^p}a^{\frac{p-2}{2}}\right]\|v\|_E^2\\
&>\frac{ma}{2}.
\end{aligned}
\]
Thus, we have
\[
\min_{r\in[0,1]}I(\gamma(r))\ge\frac{ma}{2}.
\]
Hence the minimax level satisfies $l\ge\frac{ma}{2}$.

By the mountain pass theorem, there exists a critical point $u\in S_{a,W}$ of $I$ with $I(u)=l\ge\frac{ma}{2}$. However, Lemma \ref{Lemma3.2} asserts that every critical point of $I$ on $S_{a,W}$ with $I(u)\ge\frac{ma}{2}$ is a strict local maximum, which contradicts the fact that a mountain pass critical point is not a local maximum. Therefore the two maximizers cannot be distinct, i.e., the global maximizer is unique.

Next we study the smoothness of $\phi$.
Fix an arbitrary $v_0\in S_a\cap\mathcal{V}_\chi^+$ and set $W_0=\operatorname{span}\{v_0\}$.
By Lemma~\ref{Lemma3.4} there exists a unique global maximizer $\phi(v_0)\in S_{a,W_0}$ of $I$ on $S_{a,W_0}$,
and we can write $\phi(v_0)=\tau(\eta_0)v_0+\eta_0$ with $\eta_0\in E^-$ and $\tau(\eta_0)=\sqrt{1-\frac{1}{a}\|\eta_0\|_{L^2}^2}$.
Since $\mathcal{V}_\chi^+$ is open in $E^+$ and $\{\eta\in E^-:\|\eta\|_{L^2}^2<a\}$ is open in $E^-$,
we can pick neighbourhoods $U_0\subset\mathcal{V}_\chi^+$ of $v_0$ and $V_0$ of $\eta_0$ on which the smooth local parametrisation
\[
\mathcal{P}(v,\eta):=\tau(\eta)v+\eta,\quad \tau(\eta)=\sqrt{1-\frac{1}{a}\|\eta\|_{L^2}^2},
\]
is well defined and satisfies $\mathcal{P}(v,\eta)\in S_{a,\operatorname{span}\{v\}}$ (with $\phi(v_0)=\mathcal{P}(v_0,\eta_0)$).
For $\xi\in E^-$ the corresponding tangent vector is $h_\xi=\langle d\tau(\eta),\xi\rangle v+\xi$, where $\langle d\tau(\eta),\xi\rangle=-\frac{1}{\tau(\eta)a}\Re(\eta,\xi)_{L^2}$.

To apply the implicit function theorem, we introduce the $C^1$ map
\[
F:U_0\times V_0\to(E^-)^*,\quad \langle F(v,\eta),\xi\rangle:=\langle dI(\mathcal{P}(v,\eta)),h_\xi\rangle .
\]
By construction $F(v,\eta)=0$ exactly when $\mathcal{P}(v,\eta)$ is a critical point of $I$ on $S_{a,\operatorname{span}\{v\}}$, and $F(v_0,\eta_0)=0$ thanks to Lemma~\ref{Lemma3.4}.
We now verify that $D_\eta F(v_0,\eta_0)$ has a bounded inverse. For $\xi, k \in E^-$, set
$\bar{h}_\xi = \langle d\tau(\eta_0),\xi\rangle v_0 + \xi$
and define $\bar{h}_k$ analogously. A direct computation then yields
\[
D_\eta F(v_0,\eta_0)[\xi,k]
= d^2 I(\mathcal{P}(v_0,\eta_0))[\bar{h}_\xi,\bar{h}_k] + \langle dI(\mathcal{P}(v_0,\eta_0)), d^2\tau(\eta_0)[\xi,k]v_0\rangle,
\]
where $D_\eta F(v_0,\eta_0)[\xi,k]$ denotes $(D_\eta F(v_0,\eta_0)\xi)(k)$.
Using the critical point condition $\langle dI(\mathcal{P}(v_0,\eta_0)),v_0\rangle = 2\mu(\mathcal{P}(v_0,\eta_0))\tau(\eta_0)a$ and the explicit form of $d^2\tau$, we obtain
\begin{equation}\label{3.9}
D_\eta F(v_0,\eta_0)[\xi,\xi]
\le d^2 I(\mathcal{P}(v_0,\eta_0))[\bar{h}_\xi,\bar{h}_\xi]-2\mu(\mathcal{P}(v_0,\eta_0))\|\bar{h}_\xi\|_{L^2}^2 .
\end{equation}
Substituting $\langle d\tau(\eta_0),\xi\rangle v_0 = -\frac{1}{\tau(\eta_0)a}\Re(\eta_0,\xi)_{L^2}$ and using $m\|\cdot\|_{L^2}^2 \le \|\cdot\|_E^2$, we get
\[
\|\langle d\tau(\eta_0),\xi\rangle v_0\|_E
\le \frac{\|\eta_0\|_{L^2}\|\xi\|_{L^2}}{\tau(\eta_0)a}\|v_0\|_E
\le \frac{\|\eta_0\|_{L^2}\|\xi\|_E \|v_0\|_E}{\sqrt{m} \tau(\eta_0)a}
\le \frac12 \|\xi\|_E,
\]
where the last inequality follows from \eqref{3.7} and the condition $a \in (0,a^*)$. Consequently,
\[
\frac12\|\xi\|_E \le \|\bar{h}_\xi\|_E \le \frac32\|\xi\|_E .
\]
Lemma~\ref{Lemma3.2} provides $\alpha>0$ such that
\[
d^2 I(\mathcal{P}(v_0,\eta_0))[h,h]-2\mu\|h\|_{L^2}^2 \le -\alpha\|h\|_E^2, \quad \forall h\in T_{\mathcal{P}(v_0,\eta_0)}S_{a,W_0}.
\]
Inserting $h=\bar{h}_\xi$ into \eqref{3.9} gives
\[
D_\eta F(v_0,\eta_0)[\xi,\xi] \le -\alpha\|\bar{h}_\xi\|_E^2 \le -\frac{\alpha}{4}\|\xi\|_E^2 .
\]
Define the continuous bilinear form
\[
B[\xi,k] := -D_{\eta}F(v_{0},\eta_{0})[\xi,k],\quad \xi,k\in E^{-}.
\]
The estimate above yields
\[
B[\xi,\xi] \ge \frac{\alpha}{4}\|\xi\|_{E}^{2},
\]
so $B$ is coercive. By the real Lax-Milgram theorem, for every bounded linear functional $f\in (E^{-})^{*}$ there exists a unique $\xi_{f}\in E^{-}$ such that
\[
B[\xi_{f},k] = \langle f,k\rangle, \quad \forall k\in E^{-},
\]
and $\|\xi_{f}\|_{E} \le \frac{4}{\alpha} \|f\|_{(E^{-})^{*}}$.
Because $\langle -D_{\eta}F(v_{0},\eta_{0})\xi,k\rangle = B[\xi,k]$ $\forall\xi,k\in E^{-}$, the operator $-D_{\eta}F(v_{0},\eta_{0}):E^{-}\to (E^{-})^{*}$ is bijective and the norm of its inverse is at most $\frac4\alpha$. Hence $D_{\eta}F(v_{0},\eta_{0})$ has a bounded inverse.

Thus the implicit function theorem applies: $F$ is $C^1$, $F(v_0,\eta_0)=0$, and $D_\eta F(v_0,\eta_0)$ is invertible with bounded inverse.
Consequently, there exist neighbourhoods $U\subset U_0$ of $v_0$, $V\subset V_0$ of $\eta_0$, and a unique $C^1$ map $\tilde\eta:U\to V$ such that $\tilde\eta(v_0)=\eta_0$ and $F(v,\tilde\eta(v))=0$ for all $v\in U$.
Define $\tilde\phi(v)=\mathcal{P}(v,\tilde\eta(v))=\tau(\tilde\eta(v))v+\tilde\eta(v)$.
The condition $F(v,\tilde\eta(v))=0$ means exactly that $\langle dI(\tilde\phi(v)),h_\xi\rangle=0$ for every $\xi\in E^-$, so $\tilde\phi(v)$ is a critical point of $I$ on $S_{a,\operatorname{span}\{v\}}$ and $I(\tilde\phi(v))>\frac{ma}{2}$.
By Lemma~\ref{Lemma3.2}, $\tilde\phi(v)$ is a strict local maximum.

We now identify $\tilde\phi(v)$ with the global maximizer $\phi(v)$.
The compactness and uniqueness proved in Lemma~\ref{Lemma3.4} imply that the map $v\mapsto\phi(v)$ is continuous at $v_0$.
Hence, for $v$ sufficiently close to $v_0$, the point $\phi(v)$ can be written as $\mathcal{P}(v,\eta)$ with $\eta$ belonging to the neighbourhood $V$ where the implicit function theorem applies.
By construction, $F(v,\tilde\eta(v))=0$, and writing $\phi(v)=\mathcal{P}(v,\eta)$, we also have $F(v,\eta)=0$ because $\phi(v)$ is a critical point of $I$ on $S_{a,\operatorname{span}\{v\}}$. By the local uniqueness part of the implicit function theorem, this forces $\tilde\eta(v)=\eta$.
Therefore $\tilde\phi(v)=\phi(v)$ on $U$ (after possibly shrinking $U$ further).
Thus $\phi$ is $C^1$ in a neighbourhood of $v_0$. Since $v_0$ was arbitrary, $\phi:S_a\cap\mathcal{V}_\chi^+\to E$ is $C^1$ on its whole domain.
\end{proof}
\subsection{Minimization problem}

From the above discussion, the reduced functional $\Phi_E: S_a\cap\mathcal{V}_\chi^+ \to \R$, defined by
\[
\Phi_E(v) = I(\phi(v)) = \sup_{u\in S_{a,W}} I(u), \quad W = \operatorname{span}\{v\},
\]
is well defined. The original minimax problem is thus reduced to a minimization problem:
\[
e_a = \inf_{\substack{v\in S_a\cap\mathcal{V}_\chi^+ \\ W=\operatorname{span}\{v\}}} \sup_{u\in S_{a,W}} I(u) = \inf_{v\in S_a\cap\mathcal{V}_\chi^+} \Phi_E(v).
\]

To obtain an upper bound for $e_a$, we need a more refined estimate. The next lemma provides a useful inequality.

\begin{lemma}\label{Lemma3.5}
Let $u\in H^{\frac{1}{2}}(\R^3;\C^4)$ with $\|u\|_{L^2}^2=a$. Write $u=tv+u^-$, where $u^-$ lies in the negative spectral subspace, $v$ lies in the positive spectral subspace with $\|v\|_{L^2}^2=a$, and $t=\sqrt{1-\frac{1}{a}\|u^-\|_{L^2}^2}$ so that $\|u\|_{L^2}^2=a$. Then
\[
\int_{\R^3}|x|^{-b}|u|^p dx \ge 2^{1-p} \int_{\R^3}|x|^{-b}|v|^p dx
-\frac{p\mathcal{S}}{(2\mathcal{C}_1)^p a}\|u^-\|_{L^2}^2\|v\|_E^p - \frac{\mathcal{S}}{\mathcal{C}_1^p}\|u^-\|_E^p.
\]
If, in addition, $v = U_{\mathrm{FW}}^{-1}\begin{pmatrix}
w \\
0
\end{pmatrix}$ with $w\in H^1(\R^3;\C^2)$ and $\|w\|_{L^2}^2=a$, then
\begin{equation}\label{3.10}
\begin{aligned}
\int_{\R^3}|x|^{-b}|u|^p dx &\ge 4^{1-p} \int_{\R^3}|x|^{-b}|w|^p dx - \frac{2^{1-2p}\mathcal{S}}
{m^{\frac p2}\mathcal{C}_1^p}\|\nabla w\|_{L^2}^p\\
&\quad-\frac{p\mathcal{S}}{(2\mathcal{C}_1)^p a}\|u^-\|_{L^2}^2\|v\|_E^p - \frac{\mathcal{S}}{\mathcal{C}_1^p}\|u^-\|_E^p.
\end{aligned}
\end{equation}
\end{lemma}

\begin{proof}
Since $u=tv+u^-$ and $|u|\ge t|v|-|u^-|$, the elementary inequality $(a-b)^p \ge 2^{1-p}a^p - b^p$ for $a,b\ge0$ and $p\ge1$ gives
\begin{equation}\label{3.11}
\int_{\R^3}|x|^{-b}|u|^p dx \ge 2^{1-p}t^p\int_{\R^3}|x|^{-b}|v|^p dx - \int_{\R^3}|x|^{-b}|u^-|^p dx.
\end{equation}
Using $t^p = (1-\frac{1}{a}\|u^-\|_{L^2}^2)^{\frac{p}{2}} \ge 1-\frac{p}{2a}\|u^-\|_{L^2}^2$ and the Sobolev embeddings, we obtain
\begin{equation}\label{3.12}
\int_{\R^3}|x|^{-b}|u|^p dx \ge 2^{1-p} \int_{\R^3}|x|^{-b}|v|^p dx
-\frac{p\mathcal{S}}{a( 2\mathcal{C}_1)^p }\|u^-\|_{L^2}^2\|v\|_E^p - \frac{\mathcal{S}}{\mathcal{C}_1^p}\|u^-\|_E^p.
\end{equation}
Now for $v=U_{FW}^{-1}\begin{pmatrix}
w  \\
0
\end{pmatrix}$ with $w\in H^1(\R^3;\C^2)$ and $\|w\|_{L^2}^2=a$, we have
\begin{equation*}
v(x)=\mathcal{F}^{-1}\!\left[ U^{-1}(\xi)\begin{pmatrix}\hat{w}(\xi)\\0\end{pmatrix} \right]\!(x)
= \begin{pmatrix}
\mathcal{F}^{-1}\big[ \rho_+(\xi)\hat{w}(\xi) \big](x) \\[4pt]
\mathcal{F}^{-1}\!\left[ \rho_-(\xi)\frac{\sigma\cdot\xi}{|\xi|}\hat{w}(\xi) \right]\!(x)
\end{pmatrix}.
\end{equation*}
Let $g(x)=\mathcal{F}^{-1}[\rho_+(\xi)\hat{w}]$ and $f(x)=\mathcal{F}^{-1}[(1-\rho_+(\xi))\hat{w}]$. Then $g=w-f$.  Another application of $(a-b)^p\ge 2^{1-p}a^p-b^p$ yields
\begin{equation}\label{3.13}
\int_{\R^3}|x|^{-b}|v|^p dx \ge \int_{\R^3}|x|^{-b}|g|^p dx \ge 2^{1-p}\int_{\R^3}|x|^{-b}|w|^p dx - \int_{\R^3}|x|^{-b}|f|^p dx.
\end{equation}
Furthermore, we have
\[
|1-\rho_+(\xi)|=\frac{1-\rho_+^2(\xi)}{1+\rho_+(\xi)}
=\frac{\rho_-^2(\xi)}{1+\rho_+(\xi)}\leq \rho_-^2(\xi),
\]
which yields
\[
\|f\|_E^2=\int_{\R^3}\lambda(\xi)|1-\rho_+(\xi)|^2|\hat{w}(\xi)|^2d\xi \le \int_{\R^3}\frac{1}{2}(\lambda(\xi)-m)|\hat{w}(\xi)|^2 d\xi \le \frac{1}{4m}\|\nabla w\|_{L^2}^2.
\]
Combined with the Sobolev inequalities, this implies
\begin{equation}\label{3.14}
\int_{\R^3}|x|^{-b}|f|^p dx \le \frac{\mathcal{S}}{\mathcal{C}_1^p}\|f\|_E^p
\le \frac{\mathcal{S}}
{(4m)^{\frac p2}\mathcal{C}_1^p}\|\nabla w\|_{L^2}^p.
\end{equation}
Together, \eqref{3.12}--\eqref{3.14} give \eqref{3.10}.
\end{proof}

\begin{lemma}\label{Lemma3.6}
Let $a\in(0,a^*)$. Then we have
\[
\inf_{v\in S_a\cap\mathcal{V}_{\frac\chi2}^+} \Phi_E(v) <\inf_{v\in S_a\cap(\mathcal{V}_\chi^+\setminus\mathcal{V}_{\frac{3\chi}{4}}^+)} \Phi_E(v).
\]
\end{lemma}

\begin{proof}
Fix $a\in(0,a^*)$. For any $v\in S_a\cap(\mathcal{V}_\chi^+\setminus\mathcal{V}_{\frac{3\chi}{4}}^+)$, it follows from Lemma \ref{Lemma3.4} that $\phi(v)\in S_{a,W}$ with $\Phi_E(v)=I(\phi(v))=\sup\limits_{u\in S_{a,W}}I(u)$. We have the upper bound
\begin{equation}\label{3.15}
\begin{aligned}
I(\phi(v))&=\|(\phi(v))^+\|_E^2-\|(\phi(v))^-\|_E^2-\frac{2}{p}\int_{\R^3}|x|^{-b}|\phi(v)|^p dx\\
&\le \|(\phi(v))^+\|_E^2-\|(\phi(v))^-\|_E^2 \\
&\le \|v\|_E^2-\|(\phi(v))^-\|_E^2.
\end{aligned}
\end{equation}
On the other hand,
\[
I(\phi(v)) = \sup_{u\in S_{a,W}} I(u) \ge I(v) = \|v\|_E^2 - \frac{2}{p}\int_{\R^3}|x|^{-b}|v|^p dx,
\]
which together with \eqref{3.15} gives
\begin{equation}\label{3.16}
\|(\phi(v))^-\|_E^2 \le \frac{2}{p}\int_{\R^3}|x|^{-b}|v|^p dx.
\end{equation}
Using \eqref{3.16}, $\|u\|_E^2 \ge m\|u\|_{L^2}^2$, $\|u\|_{L^p(|x|^{-b})}^p \le \mathcal{S}\|u\|_{H^{\frac{1}{2}}}^p$, and $E \cong H^{\frac{1}{2}}(\R^3;\C^4)$, we deduce that
\begin{equation}\label{3.17}
\|(\phi(v))^-\|_{L^2}^2 \le \frac{1}{m}\|(\phi(v))^-\|_E^2 \le \frac{2}{pm}\int_{\R^3}|x|^{-b}|v|^p dx \le \frac{2\mathcal{S}}{pm\mathcal{C}_1^p}\|v\|_E^p.
\end{equation}
Since $I(\phi(v))>0$, we also have $\|(\phi(v))^-\|_E^2 \le \|(\phi(v))^+\|_E^2 \le \|v\|_E^2$, whence
\begin{equation}\label{3.18}
\frac{2}{p}\int_{\R^3}|x|^{-b}|\phi(v)|^p dx \le \frac{2\mathcal{S}}{p\mathcal{C}_1^p}\|\phi(v)\|_E^p \le \frac{2\mathcal{S}}{p\mathcal{C}_1^p}
\bigl(\|v\|_E+\|(\phi(v))^-\|_E\bigr)^p \le \frac{2^{p+1}\mathcal{S}}{p\mathcal{C}_1^p}\|v\|_E^p.
\end{equation}
From \eqref{3.16}--\eqref{3.18} we therefore obtain
\[
\begin{aligned}
I(\phi(v)) &= \|(\phi(v))^+\|_E^2-\|(\phi(v))^-\|_E^2
-\frac{2}{p}\int_{\R^3}|x|^{-b}|\phi(v)|^p dx \\
&= \|v\|_E^2-\frac{1}{a}\|(\phi(v))^-\|_{L^2}^2\|v\|_E^2
-\|(\phi(v))^-\|_E^2-\frac{2}{p}\int_{\R^3}|x|^{-b}|\phi(v)|^p dx \\
&\ge \frac{3m+2}{2}a + \left[\frac12-\frac{\mathcal{S}(\frac{3m+2}{m}+2^{p})2^{\frac{p+2}{2}}(m+1)^{\frac{p}{2}}}
{p\mathcal{C}_1^p}a^{\frac{p-2}{2}}\right]a\\
&>\frac{3ma}{2}.
\end{aligned}
\]
Then we have
\begin{equation}\label{3.19}
\inf_{v\in S_a\cap(\mathcal{V}_\chi^+\setminus\mathcal{V}_{\frac{3\chi}{4}}^+)} \Phi_E(v)\ge \frac{3ma}{2}.
\end{equation}
The following is the upper bound estimate of $\inf\limits_{v\in S_a\cap\mathcal{V}_{\frac\chi2}^+} \Phi_E(v)$.
Choose $v = U_{FW}^{-1}\begin{pmatrix}
w  \\
0
\end{pmatrix} \in S_a\cap\mathcal{V}_{\frac\chi2}^+$ with $w\in H^1(\R^3;\C^2)$ and $\|w\|_{L^2}^2=a$. Then $\|v\|_{L^2}^2=a$ and $\|v\|_{H^{\frac12}}=\|w\|_{H^{\frac12}}$. Moreover,
\[
\|v\|_E^2 = \int_{\R^3}\lambda(\xi)|\hat{v}(\xi)|^2 d\xi = \int_{\R^3}\sqrt{m^2+|\xi|^2} |\hat{v}(\xi)|^2 d\xi,
\]
so
\begin{equation}\label{3.20}
0 \le \|v\|_E^2 - m\|v\|_{L^2}^2 = \int_{\R^3}(\sqrt{m^2+|\xi|^2}-m)|\hat{v}(\xi)|^2 d\xi \le \frac{1}{2m}\|\nabla w\|_{L^2}^2.
\end{equation}
For such $v$, Lemma \ref{Lemma3.4} gives $\phi(v)\in S_{a,W}$ with $\Phi_E(v)=I(\phi(v))$. Write $\phi(v)=tv+(\phi(v))^-$ where $t=\sqrt{1-\frac{1}{a}\|(\phi(v))^-\|_{L^2}^2}$. Using the estimates \eqref{3.14}, \eqref{3.17}, \eqref{3.20}, we obtain
\begin{align*}
I(\phi(v)) &= \|(\phi(v))^+\|_E^2 - \|(\phi(v))^-\|_E^2 - \frac{2}{p}\int_{\R^3}|x|^{-b}|\phi(v)|^p dx \\
&= \|v\|_E^2 - \frac{1}{a}\|(\phi(v))^-\|_{L^2}^2\|v\|_E^2 - \|(\phi(v))^-\|_E^2 - \frac{2}{p}\int_{\R^3}|x|^{-b}|\phi(v)|^p dx \\
&\le ma + \frac{1}{2m}\|\nabla w\|_{L^2}^2 + \frac{4^{1-p}\mathcal{S}}{pm^{\frac p2}\mathcal{C}_1^p}\|\nabla w\|_{L^2}^p
-\frac{2^{3-2p}}{p}\int_{\R^3}|x|^{-b}|w|^p dx\\
&\quad-\left[1-\frac{\mathcal{S}(m+1)^{\frac{p-2}{2}}}{2^{p-1}\mathcal{C}_1^p}a^{\frac{p-2}{2}}\right]
\frac{\|(\phi(v))^-\|_{L^2}^2\|v\|_E^2}{a} \\
&\quad - \left[1-\left(\frac{2\mathcal{S}}{p\mathcal{C}_1^p}\right)^{\frac p2}(m+1)^{\frac{p(p-2)}{4}}a^{\frac{p(p-2)}{4}}\right]\|(\phi(v))^-\|_E^2
.
\end{align*}
Then we have
\[
I(\phi(v)) \le ma + \frac{1}{2m}\|\nabla w\|_{L^2}^2 + \frac{4^{1-p}\mathcal{S}}{pm^{\frac p2}\mathcal{C}_1^p}\|\nabla w\|_{L^2}^p - \frac{2^{3-2p}}{p}\int_{\R^3}|x|^{-b}|w|^p dx.
\]
Now take $w_\epsilon(x)=\epsilon^{\frac{3}{2}}w(\epsilon x)$ and set $v_\epsilon = U_{FW}^{-1}\begin{pmatrix}
w_\epsilon  \\
0
\end{pmatrix}$, $W_\epsilon=span\{v_\epsilon\}$. Then we have
\[
\|\nabla w_\epsilon\|_{L^2}^2 = \epsilon^2\|\nabla w\|_{L^2}^2 \quad and  \quad \int_{\R^3}|x|^{-b}|w_\epsilon(x)|^pdx
=\epsilon^{\frac{3p}{2}-3+b}\int_{\R^3}|x|^{-b}|w|^p dx.
\]
Hence
\begin{align*}
\inf_{v\in S_a\cap\mathcal{V}_{\frac\chi2}^+} \Phi_E(v)  - ma&\le\sup\limits_{u\in S_{a,W_\epsilon}} I(u)- ma\le I(\phi(v_\epsilon)) - ma \\
&\le \frac{\epsilon^2}{2m}\|\nabla w\|_{L^2}^2 + \frac{4^{1-p}\mathcal{S}\epsilon^p}{pm^{\frac p2}\mathcal{C}_1^p}\|\nabla w\|_{L^2}^p\\
&\quad-\frac{2^{3-2p}\epsilon^{\frac{3p}{2}-3+b}}{p}\int_{\R^3}|x|^{-b}|w|^p dx.
\end{align*}
Since $p\in(2,3-b)$, we have $\frac{3p}{2}-3+b < 2$. Letting $\epsilon\to0$ shows that the right-hand side becomes negative, so we have
\[\inf\limits_{v\in S_a\cap\mathcal{V}_{\frac\chi2}^+} \Phi_E(v)  < ma,
\]
which together with \eqref{3.19} implies
\[
\inf_{v\in S_a\cap\mathcal{V}_{\frac\chi2}^+} \Phi_E(v) <\inf_{v\in S_a\cap(\mathcal{V}_\chi^+\setminus\mathcal{V}_{\frac{3\chi}{4}}^+)} \Phi_E(v).
\]
This completes the proof.
\end{proof}
We now establish the following estimates for the energy $e_a$, which will play a crucial role in the subsequent analysis.
\begin{lemma}\label{Lemma3.7}
Let $a \in (0, a^*)$. Then  $e_a\in[\frac{ma}{2},ma)$.
\end{lemma}
\begin{proof}
Since $\mathcal{V}_{\frac\chi2}^+\subset \mathcal{V}_\chi^+$, we deduce
\[
e_a= \inf_{v\in S_a\cap\mathcal{V}_\chi^+} \Phi_E(v)\le\inf_{v\in S_a\cap\mathcal{V}_{\frac\chi2}^+} \Phi_E(v) <ma.
\]
Conversely, by the definition of $e_a$ and the fact that $\frac{ma}{2} < \Psi_W=\sup\limits_{u\in S_{a,W}} I(u) \le \|v\|_E^2$, we obtain the lower bound
\[
e_a = \inf_{\substack{v\in S_a\cap\mathcal{V}_\chi^+ \\ W=\operatorname{span}\{v\}}} \sup_{u\in S_{a,W}} I(u) \ge \frac{ma}{2}.
\]
This completes the proof.
\end{proof}
\begin{proof}[\bf Proof of Theorem \ref{Th1}]
Let $\{v_n\}\subset S_a\cap\mathcal{V}_{\frac{3\chi}{4}}^+$ be a minimizing sequence for $\Phi_E$ at level $e_a$. By the Ekeland's variational principle, we obtain a new minimizing sequence $\{v_n^*\}\subset S_a\cap\mathcal{V}_{\chi}^+$ for $e_a$ such that $\|v_n^*-v_n\|_E\to0$ as $n\to\infty$ and $\{v_n^*\}$ is also a Palais-Smale sequence for $\Phi_E$ on $S_a$, i.e.,
\[
\Phi_E(v_n^*)=I(\phi(v_n^*))\to e_a,\quad \|d\Phi_E(v_n^*)\|\to 0.
\]
Set $u_n:=\phi(v_n^*)$. Since each $u_n$ is a constrained critical point of $I$ on $S_{a,W_n}$ with $W_n=\operatorname{span}\{v_n^*\}$, we have
\begin{equation}\label{3.21}
\sup_{\|h\|_{H^{\frac{1}{2}}}=1}\bigl| dI(u_n)[h]-2\mu(u_n)\Re(u_n,h)_{L^2} \bigr|\to 0.
\end{equation}
We first establish that $\{u_n\}$ is bounded in $H^{\frac{1}{2}}(\mathbb{R}^3;\mathbb{C}^4)$.
From the definition of $S_a\cap\mathcal{V}_{\chi}^+$, it follows that $\{v_n^*\}$ is bounded in $H^{\frac{1}{2}}(\mathbb{R}^3;\mathbb{C}^4)$.
Moreover, using the norm equivalence $E \cong H^{\frac{1}{2}}(\mathbb{R}^3;\mathbb{C}^4)$, the inequality $\|u_n^-\|_E \le \|u_n^+\|_E$, and the boundedness of $\|v_n^*\|_{H^{\frac{1}{2}}}$, we obtain
\[
a = \|u_n\|_{L^2}^2 \le \|u_n\|_{H^{\frac{1}{2}}}^2 \le \frac{4\mathcal{C}_2^2}{\mathcal{C}_1^2 } \|v_n^*\|_{H^{\frac{1}{2}}}^2,
\]
which implies the boundedness of $\{u_n\}$ in $H^{\frac{1}{2}}(\mathbb{R}^3;\mathbb{C}^4)$.

Repeating the argument in Lemma \ref{Lemma3.1}, we conclude that $\liminf\limits_{n\to\infty} \mu(u_n) > 0$.
Furthermore, since
\[
dI(u_n)[u_n] = 2\mu(u_n) \|u_n\|_{L^2}^2 = 2\mu(u_n) a,
\]
we have $\mu(u_n) a \le \|u_n^+\|_E^2 \le \|v_n^*\|_E^2$.
Hence the sequence $\{\mu(u_n)\}$ is bounded from above. Together with the positive lower bound, this shows that $\{\mu(u_n)\}$ is bounded, and passing to a subsequence we may assume that $\mu(u_n) \to \mu$.

Now we consider $I_\mu(u)=I(u)-\mu\|u\|_{L^2}^2$. From \eqref{3.21} we see that $\{u_n\}$ is a Palais-Smale sequence for $I_\mu$ at level $e_a$. Indeed,
\begin{align*}
\|dI_\mu(u_n)\|&=\sup_{\|h\|_{H^{1/2}}=1}\bigl|dI(u_n)[h]-2\mu(u_n)\Re(u_n,h)_{L^2}+2(\mu-\mu(u_n))\Re(u_n,h)_{L^2}\bigr|\\
&\quad\to0.
\end{align*}

We show that the limit $\mu$ satisfies $0<\mu<m$. Using the identity
\[
dI(u_n)[u_n]=2\|u_n^+\|_E^2-2\|u_n^-\|_E^2-2\int_{\mathbb{R}^3}|x|^{-b}|u_n|^p dx=2\mu(u_n)\|u_n\|_{L^2}^2+o(1),
\]
we obtain
\[
I(u_n)-\frac12 dI(u_n)[u_n]=\Bigl(1-\frac2p\Bigr)\int_{\mathbb{R}^3}|x|^{-b}|u_n|^p dx+o(1)\ge0.
\]
Thus
\[
\mu(u_n)a=I(u_n)-\Bigl(1-\frac2p\Bigr)\int_{\mathbb{R}^3}|x|^{-b}|u_n|^p dx+o(1)\le I(u_n)+o(1).
\]
Taking the limit yields $\mu a\le e_a$. On the other hand, from Lemma \ref{Lemma3.7} we have $e_a<ma$. Consequently, $\mu a<ma$, i.e., $\mu<m$. Moreover, repeating the steps of Lemma \ref{Lemma3.1}, we obtain $\mu>0$. Hence $0<\mu<m$.

Finally, we prove the strong convergence of $\{u_n\}$ in $E$. Because $\{u_n\}$ is bounded in $E$ (equivalently in $H^{\frac12}(\R^3;\C^4)$), we may assume $u_n\rightharpoonup u$ in $E$. Then it follows from Lemma \ref{Lemma2.1} that
\begin{align*}
\|u_n^+-u^+\|_E^2&=
\int_{\R^3}\bigl(|x|^{-b}|u_n|^{p-2}u_n-|x|^{-b}|u|^{p-2}u\bigr)(u_n^+-u^+) dx\\
&\quad+\langle dI_\mu(u_n)-dI_\mu(u),u_n^+-u^+\rangle+\mu\|u_n^+-u^+\|_{L^2}^2+o(1)\\
&\le \frac{\mu}{m}\|u_n^+-u^+\|_E^2+o(1),
\end{align*}
which implies $\|u_n^+-u^+\|_E^2=o(1)$. Similarly, we obtain $\|u_n^--u^-\|_E^2=o(1)$. Hence $\|u_n-u\|_E\to0$ as $n\to\infty$, which gives $\|u\|_{L^2}^2=a$. Thus $(u,\mu)\in E\times(0,m)$ is a weak normalized solution of \eqref{1.1}.
\end{proof}

% ========== 声明部分 ==========
\section*{Statements and Declarations}

\subsection*{Declaration of conflicting interests}
The authors declared no potential conflicts of interest with respect to the research, authorship, and/or publication of this article.

\subsection*{Funding}
This work is partially supported by the NSFC (12471102), NSF of Jilin Province (20250102004JC), and Research Project of the Education Department of Jilin Province (JJKH20250296KJ).

\subsection*{Data availability statement}
Data sharing not applicable to this article as no datasets were generated or analysed during
the current study.

% ========== 参考文献（Sage Vancouver 数字上标格式）==========

\end{document}